\documentclass[12pt]{amsart}
\usepackage{amssymb,verbatim,amscd,amsmath,graphicx}

\usepackage{graphicx}
\usepackage{caption}
\usepackage{subcaption}
\usepackage{pdfsync}
\usepackage{tikz}
\usepackage{pst-plot}
\usepackage{pst-node,pst-text,pst-3d,pstricks}
\usetikzlibrary{patterns}

\DeclareTextSymbolDefault{\uhorn}{T5}

\numberwithin{equation}{section}
\newtheorem{theorem}{Theorem}[section]
\newtheorem{lemma}[theorem]{Lemma}
\newtheorem{proposition}[theorem]{Proposition}
\newtheorem{corollary}[theorem]{Corollary}

\theoremstyle{definition}
\newtheorem{definition}[theorem]{Definition}
\newtheorem{def-prop}[theorem]{Definition-Proposition}

\newtheorem{remark}[theorem]{Remark}
\newtheorem{example}[theorem]{Example}

\newtheorem{problem}[theorem]{Problem}

\DeclareMathOperator{\supp}{supp}

\newcommand{\ZZ}{{\mathbb Z}}

\newcommand{\cupdot}{\mathbin{\mathaccent\cdot\cup}}

\def\e{{\bf e}}

\def\1{{\bf 1}}
\def\0{{\bf 0}}

\begin{document}


\title{Betti numbers of subgraphs}

\author{Huy T\`ai H\`a}
\address{Tulane University \\ Department of Mathematics \\
6823 St. Charles Ave. \\ New Orleans, LA 70118, USA}
\email{tha@tulane.edu}
\urladdr{http://www.math.tulane.edu/$\sim$tai/}

\author{Duc H\^o}
\address{Tulane University \\ Department of Mathematics \\
6823 St. Charles Ave. \\ New Orleans, LA 70118, USA}
\email{dho@tulane.edu}

\thanks{H\`a is partially supported by the Simons Foundation (grant \#279786). H\^o is partially supported by Tulane's Honors Summer Research grant.}

\begin{abstract}
Let $G$ be a simple graph on $n$ vertices. Let $H$ be either the complete graph $K_m$ or the complete bipartite graph $K_{r,s}$ on a subset of the vertices in $G$. We show that $G$ contains $H$ as a subgraph if and only if $\beta_{i,\alpha}(H) \le \beta_{i,\alpha}(G)$ for all $i \ge 0$ and $\alpha \in \ZZ^n$. In fact, it suffices to consider only the first syzygy module. In particular, we prove that $\beta_{1, \alpha}(H) \le \beta_{1,\alpha}(G)$ for all $\alpha \in \ZZ^n$ if and only if $G$ contains a subgraph that is isomorphic to either $H$ or a multipartite graph $K_{2, \dots, 2, a,b}$.
\end{abstract}

\maketitle


\section{Introduction} \label{sec.intro}
A graph is \emph{planar} if it can be embedded in the plane, i.e., if it can be drawn on the plane in such a way that edges do not intersect in their interiors. This class of graphs is exceptional in many ways; particularly, in the famous Four Color Theorem. Kuratowski's celebrated criterion (cf. \cite{Th}) stated that a graph $G$ is planar if and only if it does not contain any subgraph homeomorphic to $K_5$ or $K_{3,3}$. In this short note, we examine an algebraic interpretation of this criterion.

Our framework will be via the edge ideal construction. This construction gives a one-to-one correspondence between simple graphs and squarefree monomial ideals generated in degree 2. More specifically, let $G = (V,E)$ be a simple graph over the vertex set $V = \{x_1, \dots, x_n\}$. Let $k$ be a field and identify the vertices in $G$ with variables in the polynomial ring $R = k[x_1, \dots, x_n]$. The \emph{edge ideal} of $G$ is defined to be
$$I(G) = \langle x_ix_j ~|~ \{x_i,x_j\} \in E \rangle \subseteq R.$$
We shall investigate the graded Betti numbers of $I(G)$ when $G$ contains a subgraph that is homeomorphic to $K_5$ or $K_{3,3}$.

A graph that is \emph{homeomorphic} to $K_5$ or $K_{3,3}$ can be realized as a subdivision of $K_5$ or $K_{3,3}$. Here, a \emph{subdivision} of a graph results from inserting new vertices into edges. Algebraically, this corresponds to a sequence of replacing a minimal generator $xy$ of the edge ideal by two generators $xz$ and $zy$, where $z$ is a new indeterminate. The reverse-engineering of this process is simple (i.e., finding a variable $z$ that belongs to exactly two minimal generators $xz$ and $zy$, replacing these generators by $xy$, and deleting $z$ altogether). Hence, we shall focus on the study of the graded Betti numbers of $I(G)$ when $G$ contains $K_5$ or $K_{3,3}$ as a subgraph or, more generally, when $G$ contains $K_m$ or $K_{r,s}$ as a subgraph. Our results shall give algebraic interpretations of when $G$ contains $K_m$ or $K_{r,s}$ as a subgraph \emph{at specified vertices}.

Consider the naturally equipped $\ZZ^n$-graded structure of $R = k[x_1, \dots, x_n]$. For a multidegree $\alpha \in \ZZ^n$, let $\beta_{i,\alpha}(G)$ denote the $i$th $\ZZ^n$-graded Betti number of $I(G)$ in degree $\alpha$. In general, for a subgraph $H$ of $G$ it is not necessarily true that $\beta_{i,\alpha}(H) \le \beta_{i,\alpha}(G)$ for all $i \ge 0$ and $\alpha \in \ZZ^n$ (see Example \ref{ex.notSC}). Our motivating question is: for which graphs $H$ if $H$ is a subgraph of $G$ then
$$\beta_{i,\alpha}(H) \le \beta_{i,\alpha}(G) \text{ for all } i \ge 0 \text{ and } \alpha \in \ZZ^n?$$
Our results show that the complete and complete bipartite graphs belong to this class. Our first main theorem is stated as follows.

\begin{theorem}[Theorem \ref{thm.KmKrs}] \label{thm.intro1}
Let $G$ be a simple graph on $n$ vertices. Let $H$ be either the complete graph $K_m$ or the complete bipartite graph $K_{r,s}$ on a subset of the vertices in $G$. Then $G$ contains $H$ as a subgraph if and only if for all $i \ge 0$ and $\alpha \in \ZZ^n$,
\begin{align}
\beta_{i,\alpha}(H) \le \beta_{i,\alpha}(G), \tag{$\sharp$}
\end{align}
Furthermore, if $H = K_m$ then we have the equality in {\rm ($\sharp$)} whenever $\supp(\alpha)$ is a subset of the vertices in $H$.
\end{theorem}

In general, graded Betti numbers of an ideal carry rich structures and many properties of the ideal. Having $\beta_{i,\alpha}(K_{r,s}) \le \beta_{i,\alpha}(G)$ for all $i \ge 0$ and $\alpha \in \ZZ^n$ is much more than what would characterize the property that $G$ contains $K_m$ or $K_{r,s}$ as a subgraph. In fact, it suffices to consider only the first syzygy module of the edge ideals. Observe that if $G$ contains $K_m$ as a subgraph then $K_m$ is actually an \emph{induced} subgraph of $G$. As we shall see in Lemma \ref{prop.inducedSG}, in this case, multigraded Betti numbers of $K_m$ agree with corresponding multigraded Betti numbers of $G$. Furthermore, we prove the following theorem.

\begin{theorem}[Theorem \ref{thm.syzKm}] \label{thm.intro20}
Let $G$ be a simple graph on $n$ vertices. Let $H$ be the complete graph $K_m$ on a subset of the vertices in $G$. Then
$$\beta_{1,\alpha}(H) \le \beta_{1,\alpha}(G) \ \text{ for all multidegrees } \alpha \in \ZZ^n$$
if and only if $G$ contains $H = K_m$ as a subgraph.
\end{theorem}

The story for complete bipartite graphs is more subtle. The graph $G$ may contain $K_{r,s}$ as a subgraph but not as an induced subgraph. In this case, with one exception, it is still enough to consider only the first syzygy module. Our next main result is stated as follows.

\begin{theorem}[Theorem \ref{thm.syzKrs}] \label{thm.intro2}
Let $G$ be a simple graph on $n$ vertices. Let $H$ be the complete bipartite graph $K_{r,s}$ on a subset of the vertices in $G$. Then
$$\beta_{1,\alpha}(H) \le \beta_{1,\alpha}(G) \ \text{ for all multidegrees } \alpha \in \ZZ^n$$
if and only if $G$ contains either $H = K_{r,s}$ or a $K_{\underbrace{2, \dots, 2}_{t \text{ times}}, a, b}$, where $t \ge 1$ and $a+b+2t = r+s$, as a subgraph.
\end{theorem}

Our results fit well in a current on-going research program in combinatorial commutative algebra, that investigates the correspondence between algebraic invariants of squarefree monomial ideals and combinatorial structures of graphs. Work along this line includes finding algebraic algorithms to detect the existence of odd cycles in a graph (cf. \cite{FHVT1, SU}) and to detect perfect graphs (cf. \cite{FHVTperfect, SS}), and studying coloring properties of graphs via associated primes of their edge ideals (cf. \cite{FHVT2, FHVTperfect, MV}) and the packing and max-flow-min-cut properties of hypergraphs (cf. \cite{HS, HHT, HHTZ}).

The paper is outlined as follows. In the next section, we collect notation and terminology in graph theory and commutative algebra that we shall use. In particular, we recall Hochster's formula which relates multigraded Betti numbers of a squarefree monomial ideal to reduced cohomology groups of certain simplicial complexes. In Section \ref{sec.Betti}, we focus on the case where $H$ is either the complete or the complete bipartite graph, and examine an algebraic interpretation via multigraded Betti numbers of the property that $G$ contains $H$ as a subgraph. We give explicit formulae for the multigraded Betti numbers of complete and complete multipartite graphs. Our first main result, Theorem \ref{thm.intro1}, is proved in this section. The paper concludes with Section \ref{sec.syz}, where we restrict our attention to the first syzygy module of corresponding edge ideals. We prove our main results, Theorems \ref{thm.intro20} and \ref{thm.intro2}, in this section.


\section{Preliminaries}

We shall follow standard texts in commutative algebra and graph theory (cf. \cite{BM, MS, Peeva}).

\subsection{Graphs and edge ideals} Throughout the paper, $G = (V,E)$ will denote a finite \emph{simple} graph over the vertex set $V = \{x_1, \dots, x_n\}$, that is, a graph with no \emph{loops} nor \emph{multiple edges}. For simplicity of notation, we shall use the \emph{monomial} notation for edges of a graph. That is, we shall write $xy$ for the edge connecting the vertices $x$ and $y$.

\begin{definition} Let $G = (V,E)$ be a simple graph and let $l \ge 2$ be an integer.
\begin{enumerate}
\item $G$ is called a \emph{complete} graph if any pair of its vertices are connected by an edge. The complete graph on $m$ vertices is denoted by $K_m$.
\item $G$ is called a \emph{complete bipartite} graph if there is a \emph{bipartition} of the vertices in $G$, $V = X \cupdot Y$, such that edges of $G$ connect a vertex in $X$ with every vertices in $Y$, and only those, i.e.,
    $$E = \{xy ~\big|~ x \in X \text{ and } y \in Y\}.$$
    The complete bipartite graph where $|X| = r$ and $|Y| = s$ is denoted by $K_{r,s}$.
\item More generally, $G$ is called a \emph{complete $l$-partite} (or \emph{complete multipartite}) graph if there is a partition of the vertices in $G$, $V = X_1 \cupdot \dots \cupdot X_l$ such that the edges in $G$ are
    $$E = \{xy ~\big|~ x \in X_i, y \in X_j \text{ for any } 1 \le i \not= j \le l \}.$$
    The complete $l$-partite graph where $|X_t| = r_t$, for $t = 1, \dots, l$, is denoted by $K_{r_1, \dots, r_l}$.
\end{enumerate}
\end{definition}

\begin{figure}[h]
\begin{center}
\begin{tikzpicture}

\draw [line width=1pt ] (-3,3)--(-1.5,2)  ;
\draw [line width=1pt ] (-1.5,2)--(-2.25,0.5)  ;
\draw [line width=1pt ] (-2.25,0.5)--(-3.75,0.5) ;
\draw [line width=1pt ] (-3.75,0.5)--(-4.5,2)  ;
\draw [line width=1pt ] (-4.5,2)--(-3,3)  ;
\draw [line width=1pt ] (-3,3)--(-2.25,0.5)  ;
\draw [line width=1pt ] (-2.25,0.5)--(-4.5,2)  ;
\draw [line width=1pt ] (-4.5,2)--(-1.5,2) ;
\draw [line width=1pt ] (-1.5,2)--(-3.75,0.5) ;
\draw [line width=1pt ] (-3.75,0.5)--(-3,3) ;
\node at (-3,0) {$K_5$} ;

\shade [shading=ball, ball color=black]  (-3,3) circle (.1) ;
\shade [shading=ball, ball color=black]  (-1.5,2) circle (.1) ;
\shade [shading=ball, ball color=black]  (-2.25,0.5) circle (.1) ;
\shade [shading=ball, ball color=black]  (-3.75,0.5) circle (.1) ;
\shade [shading=ball, ball color=black]  (-4.5,2) circle (.1) ;

\draw [line width=1pt ] (2,3)--(2,0.5)  ;
\draw [line width=1pt ] (2,3)--(4,0.5)  ;
\draw [line width=1pt ] (2,3)--(6,0.5)  ;
\draw [line width=1pt ] (4,3)--(2,0.5)  ;
\draw [line width=1pt ] (4,3)--(4,0.5)  ;
\draw [line width=1pt ] (4,3)--(6,0.5)  ;
\draw [line width=1pt ] (6,3)--(2,0.5)  ;
\draw [line width=1pt ] (6,3)--(4,0.5)  ;
\draw [line width=1pt ] (6,3)--(6,0.5)  ;
\node at (4,0) {$K_{3,3}$} ;

\shade [shading=ball, ball color=black]  (2,0.5) circle (.1) ;
\shade [shading=ball, ball color=black]  (4,0.5) circle (.1) ;
\shade [shading=ball, ball color=black]  (6,0.5) circle (.1) ;
\shade [shading=ball, ball color=black]  (2,3) circle (.1) ;
\shade [shading=ball, ball color=black]  (4,3) circle (.1) ;
\shade [shading=ball, ball color=black]  (6,3) circle (.1) ;

\end{tikzpicture}
\end{center}
\caption{Complete and complete bipartite graphs.}\label{fig.KmKrs}
\end{figure}
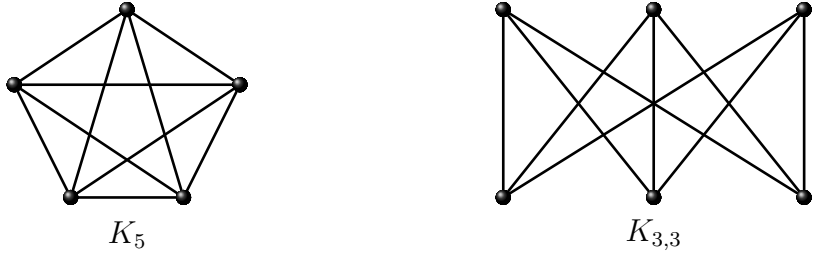

A graph $H$ is called a \emph{subgraph} of $G$ if the vertices of $H$ are vertices in $G$ and the edges of $H$ are edges in $G$.

\begin{definition} A subgraph $H$ of $G$ is called an \emph{induced} subgraph if for any two vertices $x,y$ in $H$, $xy$ is an edge in $H$ if and only if $xy$ is an edge in $G$.
\end{definition}

For a graph $G = (V,E)$, the \emph{complement} graph of $G$, denoted by $G^c$, is the graph over the same vertex set $V$, and for any $x,y \in V$, $xy$ is an edge in $G^c$ if and only if $xy$ is not an edge in $G$.

\begin{definition} Let $G = (V,E)$ be a simple graph
\begin{enumerate}
\item A collection $W \subseteq V$ of the vertices in $G$ is called an \emph{independent set} if no two vertices in $W$ are connected by an edge.
\item The \emph{independence complex} of $G$, denoted by $\Delta(G)$, is the simplicial complex over the vertex set $V$, whose faces are independent sets in $G$.
\end{enumerate}
\end{definition}

Let $k$ be a field and identify the variables of the polynomial ring $R = k[x_1, \dots, x_n]$ with the vertices in $V = \{x_1, \dots, x_n\}$.

\begin{definition} Let $G = (V,E)$ be a simple graph. The \emph{edge ideal} of $G$, denoted by $I(G)$, is defined to be the \emph{squarefree} monomial ideal
$$I(G) = \left< xy ~\big|~ xy \in E \right> \subseteq R.$$
\end{definition}

\begin{example} \label{ex.subK6}
Let $G$ be the graph given in Figure \ref{fig.edgeideal}. Then the edge ideal of $G$ is
$$I(G) = (x_1x_2, x_2x_3, x_3x_4,x_4x_5, x_1x_6) \subseteq R = k[x_1, \dots, x_6].$$
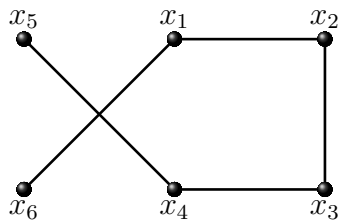
\begin{figure}[h]
\begin{center}
\begin{tikzpicture}

\draw [line width=1pt ] (0,2)--(2,2)  ;
\draw [line width=1pt ] (2,2)--(2,0)  ;
\draw [line width=1pt ] (2,0)--(0,0) ;
\draw [line width=1pt ] (0,0)--(-2,2)  ;
\draw [line width=1pt ] (0,2)--(-2,0)  ;

\shade [shading=ball, ball color=black]  (0,2) circle (.1) node [above] {$x_1$} ;
\shade [shading=ball, ball color=black]  (2,2) circle (.1) node [above] {$x_2$} ;
\shade [shading=ball, ball color=black]  (2,0) circle (.1) node [below] {$x_3$} ;
\shade [shading=ball, ball color=black]  (0,0) circle (.1) node [below] {$x_4$} ;
\shade [shading=ball, ball color=black]  (-2,2) circle (.1) node [above] {$x_5$} ;
\shade [shading=ball, ball color=black]  (-2,0) circle (.1) node [below] {$x_6$} ;

\end{tikzpicture}
\end{center}
\caption{An example of graph and its edge ideal.}\label{fig.edgeideal}
\end{figure}
\end{example}

\subsection{Multigraded Betti numbers}
Let $\e_i$ be the $i$th standard unit vector in $\ZZ^n$, for $i = 1, \dots, n$. The polynomial ring $R = k[x_1, \dots, x_n]$ is naturally equipped with a $\ZZ^n$-graded structure, given by $\deg(x_i) = \e_i$.

Let $M$ be a finitely generated $\ZZ^n$-graded $R$-module. The minimal free resolution of $M$ is of the form
$$0 \longrightarrow \bigoplus_{\alpha \in \ZZ^n} R(-\alpha)^{\beta_{p,\alpha}(M)} \longrightarrow \dots \longrightarrow \bigoplus_{\alpha \in \ZZ^n} R(-\alpha)^{\beta_{0,\alpha}(M)} \longrightarrow M \longrightarrow 0.$$
The numbers $\beta_{i,\alpha}(M)$ are called the \emph{$\ZZ^n$-graded (or multigraded) Betti numbers} of $M$.

\begin{remark} Let $G$ be a simple graph on $n$ vertices $V = \{x_1, \dots, x_n\}$. The multigraded Betti numbers of $G$ are defined to be those of its edge ideal over the corresponding polynomial ring $R = k[x_1, \dots, x_n]$. In particular, we shall often write $\beta_{i,\alpha}(G)$ for $\beta_{i,\alpha}(I(G))$.
\end{remark}

\begin{remark} For $\alpha = (\alpha_1, \dots, \alpha_n) \in \ZZ^n$, let $\supp(\alpha) = \{x_i ~|~ \alpha_i \not= 0\}$, $|\alpha| = \sum_{i=1}^n \alpha_i$, and let $x^\alpha$ denote the monomial $x_1^{\alpha_1} \dots x_n^{\alpha_n}$. We shall sometimes write $\beta_{i, x^\alpha}(G)$ in place of $\beta_{i,\alpha}(G)$, especially when it is more natural to identify the monomial $x^\alpha$. For instance, for the graph $G$ over 6 vertices $\{x_1, \dots, x_6\}$ in Example \ref{ex.subK6}, we may write $\beta_{2, x_1^5x_3^2x_6^3}(G)$ for the multigraded Betti number $\beta_{2, (5,0,2,0,0,3)}(G)$.
\end{remark}

\begin{remark} Consider a graph $H$ whose vertex set is a subset of the vertices in $G$. Assume, for simplicity, that the vertices in $H$ are $\{x_1, \dots, x_m\}$ for some $m \le n$. Let $S = k[x_1, \dots, x_m]$ be the polynomial ring corresponding to $H$. By the ring extension $S \hookrightarrow R = k[x_1, \dots, x_n]$, we can consider $I(H)$ as both an $S$-module and an $R$-module. Since $S \hookrightarrow R$ is a flat extension, the minimal free resolution of $I(H)R$ as an $R$-module is obtained by tensoring that of $I(H)$ as an $S$-module with $R$. Therefore,
$$\beta_{i, \alpha}(I(H)R) = \left\{ \begin{array}{ll}
\beta_{i,\gamma}(H) & \text{if } \gamma = (\gamma_1, \dots, \gamma_m) \in \ZZ^m, \text{ and } \alpha = (\underbrace{\gamma_1, \dots, \gamma_m}_{\gamma}, \underbrace{0, \dots, 0}_{n-m \text{ 0's}}) \\
0 & \text{otherwise.}
\end{array} \right.$$
This allows us to abuse notation and use $\beta_{i,\alpha}(H)$ to denote both $\beta_{i,\alpha}(I(H)R)$ and $\beta_{i,\gamma}(I(H))$, where $\gamma = (\gamma_1, \dots, \gamma_m) \in \ZZ^m, \text{ and } \alpha = (\underbrace{\gamma_1, \dots, \gamma_m}_{\gamma}, \underbrace{0, \dots, 0}_{n-m \text{ 0's}})$. In particular, this also makes sense of statements of the form $\beta_{i,\alpha}(H) \le \beta_{i,\alpha}(G)$ for $\alpha \in \ZZ^n$.
\end{remark}

\begin{example} \label{ex.notSC}
Consider the graph $G$ given in Example \ref{ex.subK6}. Clearly, $G$ is a subgraph of the complete graph $K_6$. Let $R = k[x_1, \dots, x_6]$ be the corresponding polynomial ring. The minimal graded free resolution of $I(G)$ is given by
$$0 \longrightarrow R(-6) \longrightarrow R^4(-5) \longrightarrow R^4(-3) \oplus R^3(-4) \longrightarrow R^5(-2) \longrightarrow I(G) \longrightarrow 0.$$
We shall also see in Lemma \ref{prop.resKm} that the edge ideal of $K_6$ has a linear resolution. Thus, the minimal free resolution of $I(G)$ is not a subcomplex of that of $I(K_6)$. In particular, graded Betti numbers of $I(G)$ are not bounded above by that of $I(K_6)$.

This is also true for multigraded Betti numbers. In particular, we have
$$\beta_{1,x_1x_4x_5x_6}(G) > \beta_{1, x_1x_4x_5x_6}(K_6).$$
This example shows that, in general, if $H$ is a subgraph of $G$, then it is not necessarily true that $\beta_{i,\alpha}(H) \le \beta_{i,\alpha}(G)$ for all $i \ge 0$ and $\alpha \in \ZZ^n$.
\end{example}

We shall often make use of the following formula of Hochster (cf. \cite[Corollary 5.12]{MS}). Here, $\widetilde{H}^\bullet$ stands for \emph{simplicial} cohomology.

\begin{proposition}[Hochster's formula] \label{prop.H}
Let $G$ be a simple graph and let $\Delta = \Delta(G)$ be its independence complex. Then the nonzero multigraded Betti numbers of $I(G)$ lie only in squarefree multidegrees $\alpha \in \{0,1\}^n$, and we have
$$\beta_{i,\alpha}(I(G)) = \dim_k \widetilde{H}^{|\alpha| - i - 2}(\Delta\big|_\alpha; k),$$
where $\Delta\big|_\alpha$ denotes the restriction of $\Delta$ on $\supp(\alpha)$, i.e.,
$$\Delta\big|_\alpha = \{\tau \in \Delta ~|~ \tau \subseteq \supp(\alpha)\}.$$
\end{proposition}


\section{Multigraded Betti numbers} \label{sec.Betti}

In this section, we shall give an algebraic characterization for the property that a graph contains a subgraph that is isomorphic to either the complete graph $K_m$ or the complete bipartite graph $K_{r,s}$. Our main theorem is stated as follows.

\begin{theorem} \label{thm.KmKrs}
Let $G$ be a simple graph on $n$ vertices. Let $H$ be either the complete graph $K_m$ or the complete bipartite graph $K_{r,s}$ on a subset of the vertices in $G$. Then $G$ contains $H$ as a subgraph if and only if for all $i \ge 0$ and $\alpha \in \ZZ^n$,
\begin{align}
\beta_{i,\alpha}(H) \le \beta_{i,\alpha}(G). \tag{$\sharp$}
\end{align}
Furthermore, if $H = K_m$ then we have the equality in {\rm ($\sharp$)} whenever $\supp(\alpha)$ is a subset of the vertices in $H$.
\end{theorem}

Before proving Theorem \ref{thm.KmKrs}, we shall need a number of auxiliary results. We first start with an observation that addresses the \emph{linear strand} of the resolution of subgraphs.

\begin{lemma} \label{prop.linear}
Let $G$ be a simple graph on $n$ vertices, and let $H$ be a subgraph of $G$. Then, for any $i \ge 0$ and $\alpha \in \ZZ^n$ such that $|\alpha| = i+2$, we have
$$\beta_{i,\alpha}(H) \le \beta_{i,\alpha}(G).$$
\end{lemma}

\begin{proof} By Hochster's formula, Proposition \ref{prop.H}, when $|\alpha| = i+2$, the multigraded Betti numbers $\beta_{i,\alpha}(H)$ and $\beta_{i,\alpha}(G)$ measure the number of connected components of $\Delta(H)\big|_\alpha$ and $\Delta(G)\big|_\alpha$. Observe further that since $H$ is a subgraph of $G$, independent subsets in $G$ are also independent subsets in $H$. Thus, $\Delta(G)\big|_\alpha$ is a subcomplex (not necessarily induced) of $\Delta(H)\big|_\alpha$. Hence, the number of connected components in $\Delta(G)\big|_\alpha$ is always bounded below by that of $\Delta(H)\big|_\alpha$.
\end{proof}

\begin{remark} Since the edge ideal $I(G)$ is generated in degree 2, the multigraded Betti numbers $\beta_{i,\alpha}(I(G))$ when $|\alpha| = i+2$, under the natural graded structure, account for linear syzygies in the resolution of $I(G)$.
\end{remark}

\begin{corollary} \label{cor.linear}
Let $G$ be a simple graph on $n$ vertices, and let $H$ be a subgraph of $G$. Suppose that, under the natural graded structure of $R = k[x_1, \dots, x_n]$, the edge ideal $I(H)$ has a linear resolution. Then, for any $i \ge 0$ and $\alpha \in \ZZ^n$, we have
$$\beta_{i,\alpha}(H) \le \beta_{i,\alpha}(G).$$
\end{corollary}

\begin{proof} Since $I(H)$ has a linear resolution, nonzero multigraded Betti numbers of $I(H)$ only appear at $\beta_{i,\alpha}(H)$ where $|\alpha| = i+2$. The conclusion is a direct consequence of Lemma \ref{prop.linear}.
\end{proof}

The next two lemmas give explicit formulae for the multigraded Betti numbers for complete and complete bipartite graphs. Similar formulae in the naturally graded case were obtained in \cite[Theorems 5.1.1 and 5.3.8]{Jacques} (see also \cite{AHH}). Our arguments for the multigraded case are along the same line, which we shall include for completeness.

\begin{lemma} \label{prop.resKm}
Let $K_n$ be the complete graphs over $n$ vertices $V = \{x_1, \dots, x_n\}$. Then, the multigraded Betti numbers of $I(K_n)$ are given as follows:
$$\beta_{i, \alpha}(K_n) = \left\{ \begin{array}{ll}
i+1 & \text{if } \alpha \in \{0,1\}^n \text{ and } |\alpha| = i+2 \\
0 & \text{otherwise.}
\end{array} \right.$$
\end{lemma}

\begin{proof}
Observe that the independence complex $\Delta = \Delta(K_n)$ consists of isolated vertices. Thus, the only nonzero reduced cohomology group of $\Delta\big|_\alpha$ is $\widetilde{H}^0(\Delta\big|_\alpha;k)$. This, together with Hochster's formula in Proposition \ref{prop.H}, implies that $\beta_{i,\alpha}(K_n) \not= 0$ only if $\alpha \in \{0,1\}^n$ and $|\alpha| = i+2$. In this case, $\Delta\big|_\alpha$ consists of $i+2$ isolated vertices, which give $i+2$ connected components. Therefore,
$$\dim_k \widetilde{H}^0(\Delta\big|_\alpha;k) = i+1,$$
and the lemma is proved.
\end{proof}

\begin{lemma} \label{prop.multi}
Let $G = K_{r_1, \dots, r_l}$ be the complete multipartite graph with $V = X_1 \cupdot \dots \cupdot X_l$ as the $l$-partition of its vertices, where $|X_t| = r_t$ for $t = 1, \dots, l$. Then, the multigraded Betti numbers of $I(G)$ are given as follows:
$$\beta_{i, \alpha}(G) = \left\{ \begin{array}{ll}
c_\alpha - 1 & \text{if } \alpha = (\gamma_1, \dots, \gamma_l) \in (0,1)^{r_1} \times \dots \times (0,1)^{r_l} \text{ and } \sum_{t=1}^l |\gamma_t| = i+2 \\
0 & \text{otherwise,}
\end{array} \right.$$
where $c_\alpha$ denotes the number of values $t$ such that $\supp(\gamma_t) \not= \emptyset$.
\end{lemma}

\begin{proof}
Observe that the independence complex $\Delta = \Delta(G)$ is the disjoint union of $l$ simplices over the vertices $X_1, \dots, X_l$. Therefore, the only nonzero reduced cohomology group of $\Delta\big|_\alpha$ is again $\widetilde{H}^0(\Delta\big|_\alpha;k)$. Coupled with Hochster's formula, Proposition \ref{prop.H}, it follows that $\beta_{i,\alpha}(G) \not= 0$ only if $\alpha = (\gamma_1, \dots, \gamma_l) \in (0,1)^{r_1} \times \dots \times (0,1)^{r_l}$ and $|\alpha| = \sum_{t=1}^l |\gamma_t| = i+2.$ In this case, $\Delta\big|_\alpha$ is the disjoint union of $c_\alpha$ simplices, and thus, has $c_\alpha$ connected components. Hence,
$$\dim_k \widetilde{H}^0(\Delta\big|_\alpha;k) = c_\alpha-1,$$
and the assertion is proved.
\end{proof}

\begin{corollary} \label{prop.resKrs}
Let $K_{r,s}$ be the complete bipartite graph with $r, s \ge 1$. Then the multigraded Betti numbers of $I(K_{r,s})$ is given as follows:
$$\beta_{i, \alpha}(K_{r,s}) = \left\{ \begin{array}{ll}
1 & \text{if } \alpha = (\gamma, \eta), \0 \not= \gamma \in (0,1)^r, \0 \not= \eta \in (0,1)^s, \text{ and } |\gamma| + |\eta| = i+2 \\
0 & \text{otherwise.}
\end{array} \right.$$
\end{corollary}

\begin{proof} The conclusion is a direct consequence of Lemma \ref{prop.multi} by taking $l = 2$.
\end{proof}

The next lemma shows that for an induced subgraph, at appropriate multidegrees, its Betti numbers agree with those of the bigger graph.

\begin{lemma} \label{prop.inducedSG}
Let $G = (V,E)$ be a simple graph over $n$ vertices and let $H = (V',E')$ be an induced subgraph of $G$ on a subset $V' \subseteq V$ of the vertices. Then for any $i \ge 0$ and $\alpha \in \ZZ^n$, we have
$$\beta_{i,\alpha}(H) \le \beta_{i, \alpha}(G).$$
Moreover, for any $i \ge 0$ and any multidegree $\alpha = (\alpha_1, \dots, \alpha_n) \in \ZZ^n$ such that $\supp(\alpha) \subseteq V'$, we have
$$\beta_{i,\alpha}(G) = \beta_{i,\alpha}(H).$$
\end{lemma}

\begin{proof} Observe that since $H$ is an induced subgraph of $G$, the independence complex $\Delta(H)$ is an induced subcomplex of the independence complex $\Delta(G)$. In particular, for any $\alpha \in \ZZ^n$, $\Delta(H)\big|_\alpha$ is an induced subcomplex of $\Delta(G)\big|_\alpha$. The first statement of the lemma thus follows from Hochster's formula in Proposition \ref{prop.H}.

For the second statement of the lemma, it suffices to observe further that if $\supp(\alpha) \subseteq V'$ then $\Delta(H)\big|_\alpha = \Delta(G)\big|_\alpha$, and therefore have the same cohomology groups. The conclusion again follows from Hochster's formula in Proposition \ref{prop.H}.
\end{proof}

We are now ready to prove our first main result, Theorem \ref{thm.KmKrs}.

\noindent{\bf Proof of Theorem \ref{thm.KmKrs}.} To prove the ($\Leftarrow$) implication, observe that the 0th Betti numbers of the edge ideal represent edges of the graph. Thus, if $\beta_{0,\alpha}(H) \le \beta_{0,\alpha}(G)$ for all $\alpha \in \ZZ^n$ then, in particular, edges of $H$ are edges in $G$, and we are done.

Let us now prove the ($\Rightarrow$) implication. It follows from Lemmas \ref{prop.resKm} and \ref{prop.resKrs} that under the natural graded structure of $R = k[x_1, \dots, x_n]$, both $I(K_m)$ and $I(K_{r,s})$ have linear resolutions. Thus, by Corollary \ref{cor.linear}, if $H$ is a subgraph of $G$ then for all $i \ge 0$ and $\alpha \in \ZZ^n$, we have
$$\beta_{i,\alpha}(H) \le \beta_{i,\alpha}(G).$$

Consider furthermore the case that $H = K_m$. As observed before, in this case, if $G$ contains $H$ as a subgraph, then $H$ is an induced subgraph of $G$. It then follows from Lemma \ref{prop.inducedSG} that for $\alpha \in \ZZ^n$ such that $\supp(\alpha)$ is a subset of the vertices in $H$, we have
$$\beta_{i,\alpha}(H) = \beta_{i,\alpha}(G) \text{ for all } i \ge 0.$$
The theorem is proved.
\qed

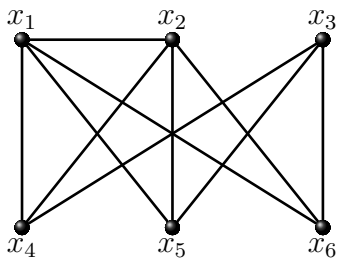
\begin{figure}[h]
\begin{center}
\begin{tikzpicture}

\draw [line width=1pt ] (2,3)--(2,0.5)  ;
\draw [line width=1pt ] (2,3)--(4,0.5)  ;
\draw [line width=1pt ] (2,3)--(6,0.5)  ;
\draw [line width=1pt ] (4,3)--(2,0.5)  ;
\draw [line width=1pt ] (4,3)--(4,0.5)  ;
\draw [line width=1pt ] (4,3)--(6,0.5)  ;
\draw [line width=1pt ] (6,3)--(2,0.5)  ;
\draw [line width=1pt ] (6,3)--(4,0.5)  ;
\draw [line width=1pt ] (6,3)--(6,0.5)  ;
\draw [line width=1pt ] (2,3)--(4,3)  ;

\shade [shading=ball, ball color=black]  (2,0.5) circle (.1) node [below] {$x_4$};
\shade [shading=ball, ball color=black]  (4,0.5) circle (.1) node [below] {$x_5$};
\shade [shading=ball, ball color=black]  (6,0.5) circle (.1) node [below] {$x_6$};
\shade [shading=ball, ball color=black]  (2,3) circle (.1) node [above] {$x_1$};
\shade [shading=ball, ball color=black]  (4,3) circle (.1) node [above] {$x_2$};
\shade [shading=ball, ball color=black]  (6,3) circle (.1) node [above] {$x_3$};

\end{tikzpicture}
\end{center}
\caption{A graph containing $K_{3,3}$ as a subgraph, having different corresponding multigraded Betti numbers compared to that of $K_{3,3}$.}\label{fig.K33}
\end{figure}

\begin{example} Let $G$ be the graph given in Figure \ref{fig.K33}. Then $G$ contains $K_{3,3}$ as a subgraph. However, we have
$$2 = \beta_{1, x_1x_2x_4}(G) > \beta_{1,x_1x_2x_4}(K_{3,3}) = 1.$$
This example illustrates that the second statement in Theorem \ref{thm.KmKrs} is not necessarily true for $H = K_{r,s}$. The reason is that a $K_m$ subgraph is an induced subgraph, while a $K_{r,s}$ subgraph needs not be. The second statement in Theorem \ref{thm.KmKrs} would be true if we assume that $G$ is a \emph{bipartite} graph.
\end{example}


\section{Multigraded first syzygies} \label{sec.syz}

In this section, we shall show that to characterize the property that a graph contains $K_m$ or $K_{r,s}$ as a subgraph, with only one exception, it suffices to consider the first syzygy module of their edge ideals.

We begin with our result for the complete graph.

\begin{theorem}\label{thm.syzKm}
Let $G$ be a simple graph on $n$ vertices. Let $H$ be the complete graph $K_m$ on a subset of the vertices in $G$. Then
$$\beta_{1,\alpha}(H) \le \beta_{1,\alpha}(G) \ \text{ for all multidegrees } \alpha \in \ZZ^n$$
if and only if $G$ contains $H = K_m$ as a subgraph.
\end{theorem}

\begin{proof} By Lemma \ref{prop.resKm}, it can be seen that $I(K_m)$ has a linear resolution under the natural graded structure of the corresponding polynomial ring. The ($\Leftarrow$) implication thus follows from Lemma \ref{prop.linear}.

To prove the ($\Rightarrow$) implication, by considering the induced subgraph of $G$ over the vertices in $H$ and making use of Lemma \ref{prop.inducedSG}, we can first assume that $G$ and $H$ are on the same vertex set (i.e. $m = n$). We shall now show that if
$$\beta_{1,\alpha}(H) \le \beta_{1,\alpha}(G) \ \text{ for all multidegrees } \alpha \in \ZZ^n$$
then $G$ is the complete graph $K_m$.

The statement can be verified directly if $n \le 2$. Assume that $n \ge 3$. Consider any two vertices $x,y$ in $G$, and suppose to the contrary that $xy$ is not an edge in $G$. Let $z$ be another vertex that is different from $x$ and $y$. By Lemma \ref{prop.resKm} and the hypothesis, we have
$$\beta_{1,xyz}(G) \ge \beta_{1,xyz}(H) = 2.$$
On the other hand, since $xy$ is not an edge in $G$, in the Taylor resolution of $I(G)$ (which is not necessarily minimal), there is at most one syzygy of degree $xyz$, which if exists necessarily comes from edges $yz$ and $xz$. This implies that
$$\beta_{1,xyz}(G) \le 1.$$
We arrive at a contradiction, and the statement is proved.
\end{proof}

Our characterization in the case for the complete bipartite graph is stated in the following theorem.

\begin{theorem} \label{thm.syzKrs}
Let $G$ be a simple graph on $n$ vertices. Let $H$ be the complete bipartite graph $K_{r,s}$ on a subset of the vertices in $G$. Then
$$\beta_{1,\alpha}(H) \le \beta_{1,\alpha}(G) \ \text{ for all multidegrees } \alpha \in \ZZ^n$$
if and only if $G$ contains either $H = K_{r,s}$ or a $K_{\underbrace{2, \dots, 2}_{t \text{ times}}, a, b}$, where $t \ge 1$ and $a+b+2t = r+s$, as a subgraph.
\end{theorem}

\begin{proof} By applying Lemma \ref{prop.inducedSG}, we may assume that $G$ and $K_{r,s}$ share the same vertex set (i.e., $n = r+s$). It can be seen from Lemma \ref{prop.multi} that the edge ideal of a multipartite graph has a linear resolution. Thus, the $(\Leftarrow)$ implication follows from Corollary \ref{cor.linear}.

To prove the $(\Rightarrow)$ implication we shall use induction on $n = r+s$. The statement is trivial for $n \le 2$. Assume that $n \ge 3$. Let $V = X \cupdot Y$ be the bipartition of the vertices in $K_{r,s}$, where $X = \{x_1, \dots, x_r\}$ and $Y = \{y_1, \dots, y_s\}$. It follows from Corollary \ref{prop.resKrs} that for any $1\le i \not= j \le r$ and $1 \le l \not= m \le s$,
$$\beta_{1, x_iy_lx_j}(K_{r,s}) \not= 0 \text{ and } \beta_{1, y_lx_jy_m}(K_{r,s}) \not= 0.$$
This implies that for any $1\le i \not= j \le r$ and $1 \le l \not= m \le s$,
\begin{align}
\beta_{1, x_iy_lx_j}(G) \not= 0 \text{ and } \beta_{1, y_lx_jy_m}(G) \not= 0. \label{eq.not0}
\end{align}

Suppose that $K_{r,s}$ is not a subgraph of $G$. Without loss of generality, assume that $x_1y_1 \not\in E$. Then, by letting $i = 1$ and $l = 1$ in (\ref{eq.not0}), we can conclude that for any $1 \le j \le r$ and any $1 \le m \le s$,  $x_1x_j, x_1y_m, y_1x_j, y_1y_m$ are edges in $G$. In particular, the vertices of $G$ can be partitioned into $V = X_1 \cupdot V'$, where $X_1 = \{x_1, y_1\}$ and $V' = V \setminus X_1$.

Now, let $G'$ be the induced subgraph of $G$ on the vertex set $V'$. Observe that the induced subgraph of $K_{r,s}$ on $V'$ is the complete bipartite graph $K_{r-1,s-1}$. Consider any $\gamma' \in \ZZ^{|V'|}$ and let $\gamma \in \ZZ^n$ be the multidegree obtained from $\gamma'$ by inserting a 0 to the coordinates corresponding to $x_1$ and $y_1$. It follows from Lemma \ref{prop.inducedSG} that
$$\beta_{1,\gamma'}(G') = \beta_{1, \gamma}(G) \ge \beta_{1, \gamma}(K_{r,s}) = \beta_{1,\gamma'}(K_{r-1,s-1}).$$
By the induction hypothesis, $G'$ contains a subgraph $H'$ that is either $K_{\underbrace{2, \dots, 2}_{t-1 \text{ times}}, a,b}$, where $2(t-1) + a + b = r+s-2$, or $K_{r-1,s-1}$.

If $H'$ is $K_{r-1,s-1}$, then since $x_1$ and $y_1$ are connected to all the vertices of $H'$, $G$ clearly contains $K_{r,s}$ as a subgraph. If $H' = K_{\underbrace{2, \dots, 2}_{t-1 \text{ times}}, a,b}$ then $G$ contains $K_{\underbrace{2, \dots, 2}_{t \text{ times}}, a,b}$ as a subgraph. The theorem is proved.
\end{proof}

Theorem \ref{thm.syzKrs} restricted to $K_{3,3}$, the case of interest for planar graphs, give the following characterization. Note that $K_{2,2,2}$ is a planar graph, while $K_{3,3}$ is not.

\begin{figure}[h]
\begin{center}
\begin{tikzpicture}

\draw [line width=1pt ] (-1,-2)--(2,0)  ;
\draw [line width=1pt ] (-1,-2)--(1,2)  ;
\draw [line width=1pt ] (-1,-2)--(-1,2)  ;
\draw [line width=1pt ] (-1,-2)--(-2,0)  ;
\draw [line width=1pt ] (1,-2)--(2,0)  ;
\draw [line width=1pt ] (1,-2)--(1,2)  ;
\draw [line width=1pt ] (1,-2)--(-1,2)  ;
\draw [line width=1pt ] (1,-2)--(-2,0)  ;
\draw [line width=1pt ] (2,0)--(-1,2)  ;
\draw [line width=1pt ] (2,0)--(-2,0)  ;
\draw [line width=1pt ] (1,2)--(-1,2)  ;
\draw [line width=1pt ] (1,2)--(-2,0)  ;

\shade [shading=ball, ball color=black]  (-1,-2) circle (.1) ;
\shade [shading=ball, ball color=black]  (1,-2) circle (.1) ;
\shade [shading=ball, ball color=black]  (2,0) circle (.1) ;
\shade [shading=ball, ball color=black]  (1,2) circle (.1) ;
\shade [shading=ball, ball color=black]  (-1,2) circle (.1) ;
\shade [shading=ball, ball color=black]  (-2,0) circle (.1) ;

\draw (0,-2) ellipse (1.5cm and 0.5cm) node {$X_1$};
\draw [rotate=60] (0.1,1.8) ellipse (1.5cm and 0.5cm) node {$X_2$};
\draw [rotate=115] (0.25,-1.85) ellipse (1.5cm and 0.5cm) node {$X_3$};

\end{tikzpicture}
\end{center}
\caption{$K_{2,2,2}$ --- a planar graph having a nonzero first Betti number at every multidegree that $K_{3,3}$ does.}\label{fig.K222}
\end{figure}
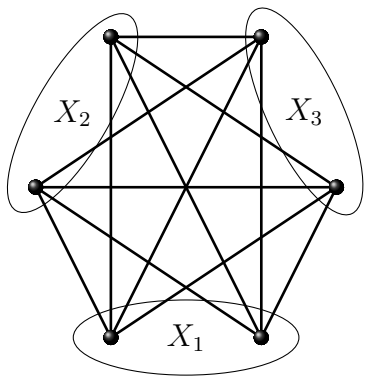

\begin{corollary} \label{cor.planar}
Let $G$ be a simple graph on $n \ge 6$ vertices. Let $H$ be the complete bipartite graph $K_{3,3}$ on a subset $W$ of 6 vertices in $G$. Then
\begin{align}
\beta_{1,\alpha}(H) \le \beta_{1,\alpha}(G) \text{ for all multidegrees } \alpha \in \ZZ^n \tag{\dag}
\end{align}
if and only if $G$ contains either $H = K_{3,3}$ or a $K_{2,2,2}$ as a subgraph. In particular, if the induced subgraph $G\big|_W$ of $G$ on $W$ is not a $K_{2,2,2}$ then $G\big|_W$ contains a $K_{3,3}$ as a subgraph if and only if {\rm ($\dag$)} holds.
\end{corollary}

\begin{proof} Notice that $K_{2,3,1}$ and $K_{2,2,1,1}$ both contain $K_{3,3}$ as a subgraph. The first statement of our assertion is a direct consequence of Theorem \ref{thm.syzKrs}. For the second statement of our assertion, it suffices to observe that adding any extra edge to a $K_{2,2,2}$ always results in a graph that contains $K_{3,3}$ as a subgraph.
\end{proof}

\begin{remark} Even though $K_{2,2,2}$ has a nonzero Betti number at every multidegree that $K_{3,3}$ does, for $|\alpha| = 3$ we in general have $\beta_{1,\alpha}(K_{2,2,2}) \not= \beta_{1,\alpha}(K_{3,3})$. Particularly, for any three distinct vertices $x_i, x_j$ and $x_k$, where $1 \le i,j,k \le 6$, we have $\beta_{1, x_ix_jx_k}(K_{3,3}) = 2$. On the other hand, in Figure \ref{fig.K222}, if $x_i, x_j, x_k$ belong to 3 distinct subsets $X_1, X_2$ and $X_3$ then $\beta_{1, x_ix_jx_k}(K_{2,2,2}) = 2$ and if two of the vertices $x_i, x_j, x_k$ belong to the same subset $X_l$, then $\beta_{1,x_ix_jx_k}(K_{2,2,2}) = 1.$
\end{remark}

Inspired by Example \ref{ex.notSC} and Corollary \ref{cor.linear}, we conclude the paper with the following problem.

\begin{problem} Let $G$ be a simple graph on $n$ vertices. Identify classes of graphs $H$ on a subset of the vertices of $G$ such that if $G$ contains $H$ as a subgraph then
$$\beta_{1,\alpha}(H) \le \beta_{1,\alpha}(G) \text{ for all } \alpha \in \ZZ^n,$$
or more generally,
$$\beta_{i,\alpha}(H) \le \beta_{i,\alpha}(G) \text{ for all } i \ge 0 \text{ and } \alpha \in \ZZ^n.$$
\end{problem}


\end{document}